\documentclass[letterpaper,11pt]{amsart}

\usepackage[all]{xy}                        %

\usepackage[margin=1in]{geometry}

\CompileMatrices                            

\UseTips                                    

\input xypic
\usepackage[bookmarks=true]{hyperref}       

\usepackage{amssymb,latexsym,amsmath,amscd}
\usepackage{xspace}
\usepackage{color}

\usepackage{graphicx}

\usepackage[utf8]{inputenc}
\usepackage{fourier}
\usepackage{array}
\usepackage{makecell}

\usepackage{array}

\reversemarginpar

\theoremstyle{plain}
\newtheorem{theorem}{Theorem}[section]
\newtheorem*{theorem*}{Theorem}
\newtheorem{proposition}[theorem]{Proposition}
\newtheorem{corollary}[theorem]{Corollary}
\newtheorem{lemma}[theorem]{Lemma}

\theoremstyle{definition}
\newtheorem{definition}[theorem]{Definition}

\newtheorem{remark}[theorem]{Remark}

\newcommand{\enm}[1]{\ensuremath{#1}}          %
\newcommand{\op}[1]{\operatorname{#1}}
\newcommand{\cal}[1]{\mathcal{#1}}

\newcommand{\CC}{\enm{\mathbb{C}}}

\newcommand{\EE}{\enm{\mathbb{E}}}

\newcommand{\QQ}{\enm{\mathbb{Q}}}

\newcommand{\ZZ}{\enm{\mathbb{Z}}}
\newcommand{\FF}{\enm{\mathbb{F}}}
\renewcommand{\AA}{\enm{\mathbb{A}}}

\newcommand{\PP}{\enm{\mathbb{P}}}

\newcommand{\Aa}{\enm{\cal{A}}}
\newcommand{\Bb}{\enm{\cal{B}}}

\newcommand{\Ee}{\enm{\cal{E}}}
\newcommand{\Ff}{\enm{\cal{F}}}
\newcommand{\Gg}{\enm{\cal{G}}}
\newcommand{\Hh}{\enm{\cal{H}}}
\newcommand{\Ii}{\enm{\cal{I}}}

\newcommand{\Ll}{\enm{\cal{L}}}

\newcommand{\Oo}{\enm{\cal{O}}}

\newcommand{\Vv}{\enm{\cal{V}}}

\renewcommand{\phi}{\varphi}
\renewcommand{\theta}{\vartheta}
\renewcommand{\epsilon}{\varepsilon}

\newcommand{\Hom}{\op{Hom}}
\newcommand{\Ext}{\op{Ext}}

\newcommand{\End}{\op{End}}

\renewcommand{\to}[1][]{\xrightarrow{\ #1\ }}

\newcommand{\old}[1]{}

\begin{document}

\title[Representation type of surfaces in $\mathbb{P}^3$]{Representation type of surfaces in $\mathbb{P}^3$}

\author{Edoardo Ballico and Sukmoon Huh}

\address{Universit\`a di Trento, 38123 Povo (TN), Italy}
\email{edoardo.ballico@unitn.it}

\address{Sungkyunkwan University, Suwon 440-746, Korea}
\email{sukmoonh@skku.edu}

\thanks{The first author is partially supported by GNSAGA of INDAM (Italy) and MIUR PRIN 2015 \lq Geometria delle variet\`a algebriche\rq. The second author is supported by the National Research Foundation of Korea(NRF) grant funded by the Korea government(MSIT) (No. 2018R1C1A6004285 and No. 2016R1A5A1008055)}


\begin{abstract}
The goal of this article is to prove that every surface with a regular point in the three-dimensional projective space of degree at least four, is of wild representation type under the condition that either $X$ is integral or $\mathrm{Pic}(X) \cong \langle \Oo_X(1) \rangle$; we construct families of arbitrarily large dimension of indecomposable pairwise non-isomorphic aCM vector bundles. On the other hand, we prove that every non-integral aCM scheme of arbitrary dimension at least two, is also very wild in a sense that there exist arbitrarily large dimensional families of pairwise non-isomorphic aCM non-locally free sheaves of rank one. 
\end{abstract}

\maketitle


\section{Introduction}
An arithmetically Cohen-Macaulay (for short, aCM) sheaf on a projective scheme $X$ is a coherent sheaf supporting $X$, which has trivial intermediate cohomology and the stalk at each point whose depth equals the dimension of $X$. ACM vector bundles correspond to maximal Cohen-Macaulay modules over the associated graded ring and they reflect the properties of the graded ring. It is believed that the category generated by aCM sheaves on $X$ measures the complexity of $X$. Indeed, a classification of aCM varieties was proposed as {\it finite, tame or wild} representation type according to the complexity of this category in \cite{DG} and there are several contributions to this trichotomy such as \cite{EH,BGS,CMP,FM}. It is only recent when such a representation type is determined for each aCM variety that is not a cone; see \cite{FP}. 

In this article, we pay our attention to the representation type of surfaces in three-dimensional projective space. Since the aCM vector bundles on smooth surfaces of degree at most two are completely classified due to the work by Horrocks and \cite{Kapranov, Knorrer}, we may focus on surfaces of degree at least three. The case of cubic surfaces is dealt in \cite{CH, faenzi} and the case of quartic surfaces is from \cite{madonna}. Our main result is the following, which implies that the surfaces in Theorem \ref{thth} are of wild representation type. 

\begin{theorem}\label{thth}
Let $X\subset \PP^3$ be a surface of degree at least four with $X_{\mathrm{reg}} \ne \emptyset$ and assume either $\mathrm{Pic}(X) =\ZZ\langle \Oo_X(1)\rangle$ or that $X$ is integral. For every even and positive integer $r$, there exists a family $\{\Ee _\lambda\}_{\lambda \in \Lambda}$ of indecomposable aCM vector bundles of rank $r$ such that $\Lambda$ is an integral quasi-projective variety with $\dim \Lambda=r$ and $\Ee _{\lambda}\ncong \Ee _{\lambda'}$ for all $\lambda \ne \lambda'$ in $\Lambda$.
\end{theorem}

It has to be noticed that although the result in \cite{FP} is more general than the implication of Theorem \ref{thth} regarding the wildness of the representation type, Theorem \ref{thth} provides a concrete way of constructing families of indecomposable aCM `vector bundles' with prescribed rank, even on singular surfaces. 

On the other hand, every non-integral aCM projective schemes of arbitrary dimension at least two is of `very wild' representation type, in a sense that there exist arbitrarily large dimensional families of pairwise non-isomorphic aCM non-locally free sheaves of rank one; see Proposition \ref{zz3}. 

Here we summarize the structure of this article. In Section \ref{prem} we collect several definitions and basic results that are used throughout the article. In Section \ref{sec3} we state the main result in Theorem \ref{i1}, which would automatically imply Theorem \ref{thth}. We also give a proof of Theorem \ref{i1} in special case and suggest a number of its variation to construct aCM vector bundles. Then we spend the whole Section \ref{sec4} for the proof of Theorem \ref{i1}; basically we use induction on rank and the main ingredient for the proof is Lemma \ref{x1} and the use of monodromy argument. Then we show in Section \ref{sec5} the wildness of any aCM projective scheme of dimension at least two by investigating non-locally free ideal sheaves. 


\section{Preliminary}\label{prem}
Throughout the article our base field $\mathbf{k}$ is algebraically closed of characteristic $0$. We always assume that our projective schemes $X\subset \PP^N$ are arithmetically Cohen-Macaulay, namely, $h^1(\Ii _{X, \PP^N}(t)) =0$ for all $t\in \ZZ$ and $h^i(\Oo _X(t)) =0$ for all $t\in \ZZ$ and all $i=1,\dots ,\dim X-1$, of pure dimension at least two. Then by \cite[Th\'eor\`eme 1 in page 268]{serre} all local rings $\Oo _{X,x}$ are Cohen-Macaulay of dimension $\dim X$. From $h^1(\Ii _{X, \PP^N})=0$ we see that $X_{\mathrm{red}}$ is connected. Since in all our results we have $N =\dim X+1=3$, the reader may just assume that $X$ is a surface in $\PP^3$. For a vector bundle $\Ee$ of rank $r\in \ZZ$ on $X$, we say that $\Ee$ {\it splits} if all its indecomposable factors are $\Oo_X(t)$ for some $t\in \ZZ$; $\Ee \cong \oplus_{i=1}^r \Oo_X(t_i)$ for some $t_i\in \ZZ$ with $i=1,\ldots, r$.

We always fix the embedding $X\subset \PP^N$ and the associated polarization $\Oo_X(1)$. For a coherent sheaf $\Ee$ on a closed subscheme $X$ of a fixed projective space, we denote $\Ee \otimes \Oo_X(t)$ by $\Ee(t)$ for $t\in \ZZ$. For another coherent sheaf $\Gg$, we denote by $\mathrm{hom}_X(\Ff, \Gg)$ the dimension of $\Hom_X(\Ff, \Gg)$, and by $\mathrm{ext}_X^i(\Ff, \Gg)$ the dimension of $\mathrm{Ext}_X^i (\Ff, \Gg)$. Finally we denote the canonical sheaf of $X$ by $\omega_X$. 

\begin{definition}\label{deff}
A coherent sheaf $\Ee$ on $X$ is called {\it arithmetically Cohen-Macaulay} (for short, aCM) if the following hold:
\begin{itemize}
\item [(i)] $\Ee$ is locally Cohen-Macaulay, that is, the stalk $\Ee_x$ has depth equal to $\dim \Oo_{X,x}$ for any point $x$ on $X$, and
\item [(ii)] $H^i(\Ee(t))=0$ for all $t\in \ZZ$ and $i=1, \ldots, \dim (X)-1$.
\end{itemize}
\end{definition}

\begin{remark}\label{rremm}
In the condition (i) of Definition \ref{deff}, we may only require that the stalk $\Ee_x$ has positive depth for any point $x\in X$; see \cite[Remark 2.2]{BHMP} and \cite[Th\'eor\`eme 1 in page 268]{serre}. 
\end{remark}

If $\Ee$ is a coherent sheaf on a closed subscheme $X$ of a fixed projective space, then we may consider its Hilbert polynomial $\mathrm{P}_{\Ee}(t)\in \QQ [t]$ with the leading coefficient $\mu(\Ee)/d!$, where $d$ is the dimension of $\mathrm{Supp}(\Ee)$ and $\mu=\mu(\Ee)$ is called the {\it multiplicity} of $\Ee$. The {\it normalized} Hilbert polynomial $p_{\Ee}(t)$ of $\Ee$ is defined to be the Hilbert polynomial of $\Ee$ divided by $\mu (\Ee)$.

\begin{definition}\label{def}
If $\dim \mathrm{Supp}(\Ee)=\dim (X)$, then the {\it rank} of $\Ee$ is defined to be
$$\mathrm{rank}(\Ee)=\frac{\mu(\Ee)}{\mu(\Oo_X)}.$$
Otherwise it is defined to be zero.
\end{definition}
For an integral scheme $X$, the rank of $\Ee$ is the dimension of the stalk $\Ee_x$ at the generic point $x\in X$. But in general $\mathrm{rank}(\Ee)$ needs not be integer.

\begin{lemma}\label{z1}
Let $(X, \Oo_X(1))$ be an aCM projective scheme of dimension $n\ge 2$. For a fixed coherent sheaf $\Gg$ with pure depth $n$ on $X$, assume the existence of $t_0\in \ZZ$ such that $s:= h^1(\Gg (t_0))> 0$. Then the vector space $W:= H^1(\Gg (t_0))$ induces the following unique extension up to isomorphisms
\begin{equation}\label{eqz1}
0 \to \Gg \to \Ee \to \Oo _X(-t_0)\otimes W^\vee \to 0
\end{equation}
and the sheaf $\Ee$ in the middle satisfies the following:
\begin{itemize}
\item [(i)] $h^1(\Ee (t)) = h^1(\Gg (t))$ for all $t\ne t_0$, and $h^1(\Ee (t_0)) =0$;
\item [(ii)] $h^i(\Ee (t)) =h^i(\Gg (t))$ for all $t\in \ZZ$ and all $i$ with $2\le i \le n-1$.
\end{itemize}
If $\Gg$ is locally free, then $\Ee$ is locally free.
\end{lemma}

\begin{proof}
All statements, except the one concerning $h^1(\Ee (t_0))$, are true for any sheaf $\Ee$ fitting into (\ref{eqz1}). The vanishing
of $H^1(\Ee (t_0))$ is equivalent to the bijectivity of the coboundary map $\delta : H^0(\Oo _X)\otimes W^\vee \rightarrow H^1(\Gg (t_0))$ associated to the twist by $\Oo _X(t_0)$ of (\ref{eqz1}). The bijectivity of $\delta$ is a standard result on the extension functor.
\end{proof}

\begin{theorem}\label{e2}
Let $X\subset \PP^N$ be a projective Gorenstein scheme with pure dimension two and pure depth two, satisfying that 
\begin{itemize}
\item $h^1(\Oo_X(t))=0$ for all $t\in \ZZ$ and $h^1(\Ii_{X,\PP^N})=0$; 
\item $X_{\mathrm{reg}} \ne \emptyset$ and $\deg (\omega _X)+\deg (X) \ge 0$.
\end{itemize}
Then there exists a two-dimensional family of pairwise non-isomorphic aCM vector bundles of rank two on $X$ whose very general member is indecomposable; here ``very general" means outside countably many proper subvarieties.
\end{theorem}

\begin{proposition}\label{e1}
Let $X\subset \PP^N$ be as in Theorem \ref{e2}. Assume $X_{\mathrm{reg}} \ne \emptyset$ and fix $p\in X_{\mathrm{reg}}$. Then there exists an aCM vector bundle $\Ee _p$ of rank two on $X$ fitting into the exact sequence
\begin{equation}\label{eqe1}
0 \to \omega _X(1)\to \Ee _p\to \Ii _{p,X}\to 0.
\end{equation}
Moreover, if $\deg (\omega _X)+\deg (X) \ge 0$ and $p,q\in X_{\mathrm{reg}}$ with $p\ne q$, then we have $\Ee _p\ncong \Ee  _q$.
\end{proposition}

\begin{proof}
Since $X$ is Gorenstein, $\omega _X(1)$ is a line bundle and we get
$$\Ext_X^1(\Ii _{p,X},\omega _X(1)) \cong H^1(\Ii _{p,X}(-1))^\vee \cong \mathbf{k}.$$ 
So up to isomorphism there exists a unique sheaf $\Ee_p$ fitting into an extension (\ref{eqe1}) with a nonzero extension class. Since $h^0(\Oo _X(-1)) =0$ and $p\in X_{\mathrm{reg}}$, the Cayley-Bacharach condition is satisfied for (\ref{eqe1}) and so $\Ee_p$ is locally free; see \cite{cat}. Note that the restriction map 
$$H^0(\Oo _X(t)) \to H^0(\Oo _X(t)_{|\{p\}})$$
is surjective for any $t\ge 0$. This implies that $h^1(\Ii_{p,X}(t))=0$ for any $t\ge 0$, because we have $h^1(\Oo_X(t))=0$. Then we see from (\ref{eqe1}) that $h^1(\Ee_p(t))=0$ for any $t\ge 0$. On the other hand, from $\det (\Ee_p) \cong \omega_X(1)$, we get that $h^1(\Ee_p(t))=h^1(\Ee_p^\vee\otimes \omega_X(-t))=h^1(\Ee_p(-t-1))=0$ for $t<0$ by Serre's duality. Thus $\Ee_p$ is aCM. 

For the second assertion, assume $\Ee _p\cong \Ee _q$. From the assumption $\deg (\omega _X(1)) \ge 0$, we get $h^0(\omega _X^\vee (-1)) \le 1$ with equality if and only if $\omega _X \cong \Oo _X(-1)$. In particular, we have $h^0(\Ii _{p,X}\otimes \omega _X^\vee (-1)) =0$. Then from the assumption $h^1(\Oo _X)=0$ and (\ref{eqe1}), we get $h^0(\Ee _p\otimes \omega _X^\vee (-1)) =1$ and that $p$ is the only zero of a nonzero section of $H^0(\Ee _p\otimes \omega _X^\vee (-1))$. Thus we get $p=q$.
\end{proof}

\begin{proof}[Proof of Theorem \ref{e2}:]
By assumption $X_{\mathrm{reg}}$ is a two-dimensional quasi-projective smooth variety. By Proposition \ref{e1} there is a flat family of aCM vector bundles $\{\Ee _p\}_{p\in X_{\mathrm{reg}}}$ of rank two such that if $p, q\in X_{\mathrm{reg}}$ and $p\ne q$, then $\Ee _p\ncong \Ee _q$. Now assume that $\Ee_p$ is decomposable for some $p\in X_{\mathrm{reg}}$, say $\Ee _p\cong \Aa_1 \oplus \Aa_2$ with each $\Aa_i$ a line bundle on $X$. Since $\det (\Ee _p)\cong \omega _X(1)$, we have $\Aa_2 \cong \Aa_1^\vee\otimes \omega _X(1)$. Now from the assumption that $h^1(\Oo _X)=0$, we see that $\mathrm{Pic}(X)$ is discrete and countable. This implies that there can exist only countably many decomposable vector bundles in the family. Since the base field $\mathbf{k}$ is algebraically closed and so uncountable, there exists some indecomposable vector bundle in the family $\{\Ee _p\}_{p\in X_{\mathrm{reg}}}$ and for a very general point $o$ on any connected component of $X_{\mathrm{reg}}$ the vector bundle $\Ee _o$ is indecomposable.
\end{proof}

Throughout the article, as in Proposition \ref{e1}, our construction of aCM sheaf of rank two on $X$ is in terms of the following extension
\begin{equation}\label{popo}
0\to \omega_X \to \Ee \to \Ii_{Z,X}(a) \to 0
\end{equation}
with $Z$ a locally complete intersection of codimension two in $X$ and $a\in \ZZ$. Such extensions are parametrized by $\mathrm{Ext}_X^1(\Ii_{Z,X}(a), \omega_X)$. In case when $X$ is a surface, the coboundary map associated to (\ref{popo}) is
$$\delta_1 : H^1(\Ii_{Z,X}(a)) \to H^2(\omega_X)\cong \mathbf{k}$$ 
and by Serre's duality in \cite[Theorem 3.12]{H} its dual is
$$\mathbf{k}\cong \mathrm{Hom}_X(\omega_X, \omega_X) \to \mathrm{Ext}_X^1(\Ii_{Z,X}(a), \omega_X),$$
which is obtained by applying the functor $\mathrm{Hom}_X(-, \omega_X)$ to (\ref{popo}). Thus the coboundary map $\delta_1$ is surjective if and only if (\ref{popo}) is a non-trivial extension. Since we assume $h^1(\Oo_X)=h^1(\omega_X)=0$, this implies that $h^1(\Ee)=h^1(\Ii_{Z,X}(a))-1$.

\section{aCM vector bundle on surfaces in $\PP^3$}\label{sec3}
We always assume that $X\subset \PP^3$ is a surface of degree $m$, not necessarily smooth. In particular, its dualizing sheaf is $\omega_X \cong \Oo_X(m-4)$ and we get $h^2(\Oo_X)=\binom{m-1}{3}$. We also have $h^0(\Oo _X)=1$ and $h^1(\Oo _X)=0$. 

\begin{lemma}\label{c2}
Each line bundle $\Oo _X(t)$ with $t \in \ZZ$, is stable as an $\Oo _{\PP^3}$-sheaf with pure depth $2$.
\end{lemma}

\begin{proof}
It is enough to deal with the case $t=0$. Assume the contrary and take a subsheaf $\Aa \subsetneq \Oo _X$ such that $\Bb:=\Oo _X/\Aa$ has depth $2$ and normalized Hilbert polynomial at least the one of $\Oo _X$. Since $\Bb$ is a quotient of $\Oo _X$ with depth $2$ and $X$ has no embedded component, we get $\Bb \cong \Oo _T$ for $T$ a union of some of the irreducible components of $X_{\mathrm{red}}$ with at most the multiplicities appearing in $X$. This implies that $T\in |\Oo _{\PP^3}(d)|$ for some integer $d$ with $1\le d < m$. Now the Hilbert polynomial of $\Oo _X$ is 
\begin{align*}
\mathrm{P}_{\Oo_X}(t)&=\binom{t+3}{3} -\binom{t-m+3}{3} \\
&= \left (\frac{m}{2}\right) t^2 + \left ( 2m-\frac{m^2}{2}\right) t + \left(\frac{m^3}{6} -m^2 + \frac{11m}{6} \right). 
\end{align*}
Similarly, we get the Hilbert polynomial $\mathrm{P}_{\Oo_T}(t)$ of $\Oo_T$ by replacing $m$ in $\mathrm{P}_{\Oo_X}(t)$ by $d$. Then we see that $p_{\Oo_X}(t)<p_{\Oo _T}(t)$ for $t\gg 0$, a contradiction.
\end{proof}

\begin{remark}\label{c3}
If either $\mathrm{Pic}(X) \cong \ZZ \langle \Oo_X(1)\rangle$ or $X$ is integral, then every line bundle is stable. Note also that the proof of Lemma \ref{c2} shows that the ideal sheaf $\Ii _{Z,X}$ for any zero-dimensional subscheme $Z\subset X$, is also stable. If $X$ is integral, then any sheaf of rank $1$ with positive depth is stable. Thus these sheaves are indecomposable.
\end{remark}

\begin{proposition}\label{a1.1}
Let $X\subset \PP^3$ be a surface of degree $m\ge 2$ with $X_{\mathrm{reg}} \ne \emptyset$. Fix $p \in X_{\mathrm{reg}}$, and let $\Ee_p$ be the unique non-trivial extension  
\begin{equation}\label{eqb1}
0 \to \Oo _X(m-3) \to \Ee_p \to \Ii _{p,X}\to 0.
\end{equation}
Then $\Ee_p$ is an aCM vector bundle of rank two on $X$ and $\Ee \ncong \Oo _X(a)\oplus \Oo _X(b)$ for any $a, b\in \ZZ$. If one of the following holds, then $\Ee$ is indecomposable. 
\begin{itemize}
\item [(i)] $\mathrm{Pic}(X) \cong \ZZ\langle \Oo _X(1)\rangle $,
\item [(ii)] $\Oo _X(t)$ for $t\in \ZZ$ are the only aCM line bundles on $X$, or 
\item [(iii)] $m\ge 4$ and $X$ is integral.
\end{itemize}
\end{proposition}

\begin{proof}
By Proposition \ref{e1} it remains to deal with indecomposability of $\Ee_p$. First show that there are no integers $a, b$ such that $\Ee_p \cong \Oo _X(a)\oplus \Oo _X(b)$. Assume that such $a,b$ exist, say $a\ge b$. Since $h^0(\Ee_p (3-m)) =1$ and $h^0(\Ee_p (2-m)) =0$, we get $(a,b)=(m-3,0)$ and $m\ge 3$. Then we get $h^0(\Ee_p) = \binom{m}{3} +1$, while (\ref{eqb1}) gives $h^0(\Ee_p )=\binom{m}{3}$. 

Now assume that $\Ee_p$ is decomposable. Since $\Ee_p$ is locally free and it has rank $2$, we have $\Ee_p \cong \Aa_1 \oplus \Aa_2$ with each $\Aa_i\in \mathrm{Pic}(X)$. Since $\Ee_p$ is aCM, each $\Aa_i$ is aCM. In cases (i) and (ii) the assertion holds by above. Thus we assume the case (iii). By Lemma \ref{c2} and Remark \ref{c3}, (\ref{eqb1}) is the HN filtration of $\Ee_p$. Applying the functor $\mathrm{Hom}_X(\Ee_p, -)$ to (\ref{eqb1}), we get 
$$0\to \mathrm{Hom}_X(\Ee_p, \Oo_X(m-3)) \to \mathrm{Hom}_X(\Ee_p, \Ee_p) \to \mathrm{Hom}_X(\Ee_p, \Ii_{p,X})\to \mathrm{Ext}_X^1(\Ee_p, \Oo_X(m-3)).$$
Note that $\hom_X(\Ee_p, \Oo_X(m-3))=h^2(\Ee_p(-1))=h^0(\Ee_p)=\binom{m}{3}$ by Serre's duality. By applying the functor $\mathrm{Hom}_X(-, \Ii_{p,X})$ to (\ref{eqb1}), we get 
$$\hom_X(\Ee_p, \Ii_{p,X})=\hom_X(\Ii_{p,X}, \Ii_{p,X})=1.$$ 
Thus we have 
$$\binom{m}{3}\le \hom_X (\Ee_p ,\Ee_p )\le 1 + \binom{m}{3}.$$
Since $h^0(\Oo _X) =1$, we have $\hom_X (\Aa_i ,\Aa_i ) = 1$ for each $i$. So we get
$$\hom_X (\Ee_p ,\Ee_p ) =2+\hom_X (\Aa_1 ,\Aa_2 )+\hom_X (\Aa_2,\Aa_1 ).$$ 
Since $X$ is integral, each $\Aa_i$ is stable and we get either $\Aa_1\cong \Aa_2$ or $\hom_X (\Aa_i ,\Aa_{3-i}) =0$ for each $i$. In the latter case we have $\hom_X(\Ee_p, \Ee_p)=2<\binom{m}{3}$, a contradiction. In the former case, we have $\mathrm{hom}_X(\Ee_p, \Ee_p)=4$ and the only possibility is $m=4$. But this is also impossible, since we would get $\Aa_1^{\otimes 2} \cong \det (\Ee_p) \cong \Oo_X(1)$.  
\end{proof}

\begin{proposition}\label{z2}
Let $X\subset \PP^3$ be a surface of degree $m\ge 2$ and let $Z\subset X$ be a zero-dimensional subscheme of degree $3$, which is not collinear. Assume that $Z$ is a locally complete intersection inside $X$, i.e. for each $p\in Z_{\mathrm{red}}$ the ideal sheaf of $Z$ at $\Oo _{X,p}$ is generated by two elements of $\Oo _{X,p}$. Then there is a vector bundle $\Gg$ of rank two fitting into an exact sequence
\begin{equation}\label{eqz2}
0 \to \Oo _X(m-4)\to \Gg \to \Ii _{Z,X}\to 0
\end{equation}
with $h^1(\Gg (t)) =0$ for all $t\ne 0$ and $h^1(\Gg)=1$. There is also an exact sequence
\begin{equation}\label{eqz3}
0 \to \Gg \to \Ee \stackrel{u}{\to} \Oo _X\to 0,
\end{equation}
where $\Ee$ is an aCM vector bundle of rank three such that $\Ee \ncong \Oo _X(a_1)\oplus
\Oo _X(a_2)\oplus \Oo _X(a_3)$ for any $(a_1, a_2, a_3)\in \ZZ^{\oplus 3}$. Moreover, if $\mathrm{Pic}(X) \cong \ZZ \langle \Oo_X(1) \rangle$, then $\Ee$ is indecomposable. 
\end{proposition}

\begin{proof}
Since $\omega _X\cong \Oo _X(m-4)$, we have $h^0(\Oo _X(4-m)\otimes \omega _X)=1$ and $\Oo _X(4-m)\otimes \omega _X$ is globally generated. Since $\Oo _X(4-m)\otimes \omega _X$ is globally generated, we have $h^0(\Ii _{p,X}\otimes \Oo _X(4-m)\otimes \omega _X)=0$ for all $p\in Z_{\mathrm{red}}$. Since $Z$ is a locally complete intersection, the Cayley-Bacharach condition is satisfied and so there is a locally free $\Gg$ fitting into (\ref{eqz2}); see \cite{cat}. From (\ref{eqz2}) we immediately get $h^1(\Gg (t)) =0$ for all $t>0$, because $Z$ is not collinear. Note that $\det (\Gg ) \cong \Oo _X(m-4)$ and $\Gg$ is a vector bundle of rank two. This implies $\Gg ^\vee \cong \Gg (4-m)$. For $t<0$, we have $h^1(\Gg (t)) = h^1(\Gg ^\vee (-t)\otimes \omega _X) = h^1(\Gg (-t)) =0$ by Serre's duality. Now consider the coboudnary map $\delta _1: H^1(\Ii _{Z,X}) \rightarrow H^2(\Oo _X(m-4))\cong \mathbf{k}$ with $\mathrm{ker}(\delta _1) =H^1(\Gg)$. The dual of $\delta_1$ is the map 
$$\mathrm{Hom}_X(\Oo_X(m-4), \Oo_X(m-4)) \to \mathrm{Ext}_X^1(\Ii_{Z,X}, \Oo_X(m-4)) $$
sending the identity map to the element corresponding to $\Gg$. This implies that $\delta_1$ is surjective and $h^1(\Gg)=1$. 

Now we apply Lemma \ref{z1} to $\Gg$ to obtain an aCM vector bundle $\Ee$ of rank three fitting into (\ref{eqz3}). Since $h^1(\Gg )=1$ and $h^1(\Ee )=0$, (\ref{eqz2}) and (\ref{eqz3}) give $h^0(\Ee )=h^0(\Gg )=\binom{m-1}{3}$. Assume the existence of integers $a_1\ge a_2\ge a_3$ such that $\Ee \cong \oplus_{i=1}^3 \Oo _X(a_i)$. Since $\det (\Ee )\cong \Oo _X(m-4)$, we have $a_1+a_2+a_3 =m-4$. If $2 \le m\le 3$, then we have $a_1\ge 0$ from $a_1+a_2+a_3 =m-4$. This implies that $h^0(\Oo _X(a_1)) > 0 =\binom{m-1}{3}=h^0(\Ee)$, a contradiction. If $m=4$, then we have $h^0(\Ee )=1$. Since $a_1+a_2+a_3 =0$, we have $\sum _{i=1}^{3} h^0(\Oo _X(a_i)) >1$, a contradiction. Finally assume $m>4$. From (\ref{eqz2}) and (\ref{eqz3}) we see that $\Oo _X(m-2)$ is the first non-trivial sheaf in the HN filtration of $\Ee$. Thus $a_1 =m-4$ and $h^0(\Oo _X(a_1)) =\binom{m-1}{3}$. Since $a_2+a_3=0$, we have $h^0(\Oo _X(a_2)) >0$ and so $h^0(\Ee )> \binom{m-1}{3}$, a contradiction. Hence we get $\Ee \ncong \oplus_{i=1}^3\Oo_X(a_i)$ for any triple of integers $(a_1, a_2, a_3)$. 

It remains to show the last assertion. Assume $\mathrm{Pic}(X) \cong\ZZ\langle \Oo _X(1)\rangle$ and that $\Ee$ is decomposable; by the previous assertion we have $\Ee \cong \Aa_1 \oplus \Aa_2$ with $\mathrm{rank}(\Aa_i ) =i$ for each $i$ and $\Aa_2$ indecomposable. Set $\Aa_1 \cong \Oo _X(a)$ for $a\in \ZZ$. Since $h^0(\Ee )=\binom{m-1}{3}$, we have $a\le m-4$. From (\ref{eqz2}) and (\ref{eqz3}) we get the existence of a subsheaf $\Ff \subset \Ee$ such that $\Ff \cong \Oo _X(m-4)$ and $\Ee /\Ff$ is an extension $\Hh$ of $\Oo _X$ by $\Ii _{Z,X}$. Note that $\Hh$ is not locally free, because $\Ii _{Z,X}$ has not depth $2$. In particular, $\Hh$ is not isomorphic to $\Aa_2$ and we get $\Aa_1 \ncong \Ff$. So we have $a<m-4$. Now consider a restriction map
$$u_{|\{0\}\oplus \Aa_2} : \{0\}\oplus \Aa_2 \to \Oo_X.$$
If this restriction map is surjective, then its kernel is a line bundle, say $\Oo_X(b)$. Since $X$ is aCM, we get $\Aa_2 \cong \Oo_X\oplus \Oo_X(b)$, a contradiction. Thus the restriction map is not surjecitve and so the other restriction map $u_{|\Aa_1 \oplus \{0\}} $ is not zero. In particular, we get $a\le 0$. If $a=0$, then we have $\Aa_1 \cong \Oo_X$ and the map $u_{|\Aa_1 \oplus \{0\}}$ is an isomorphism. Thus (\ref{eqz3}) splits and we get $h^1(\Ee )\ge h^1(\Gg )>0$, a contradiction. Hence we get $a<0$. Since there is no nonzero map $\Ff \rightarrow \Aa_1$ from $a<m-4$, $\Ff$ is isomorphic to a subsheaf $\Ff_1$ of $\Aa_2$ and we get $\Hh \cong \Oo_X(a)\oplus \Aa_2/\Ff_1$. From $a<0$ we see that there is no nonzero map $\Ii_{Z,X} \rightarrow \Oo_X(a)$. Since $\Hh$ is an extension of $\Oo_X$ by $\Ii_{Z,X}$, we get that $\Ii_{Z,X} \cong \Aa_2/\Ff_1$ and so $\Oo_X(a)\cong \Oo_X$, a contradiction. 
\end{proof}

\begin{remark}
In case $m=1$, i.e. $X=\PP^2$, we fail in obtaining an indecomposable aCM vector bundle of rank three, using the method in Proposition \ref{z2}. Indeed, we get $\Gg \cong \Omega_{\PP^2}^1$ and the corresponding  vector bundle of rank three is $\Ee \cong \Oo _{\PP^2}(-1)^{\oplus 3}$.
\end{remark}

\begin{corollary}
Let $X\subset \PP^3$ be union of multiple planes in which at least one plane occurs with multiplicity $1$. Then there is an indecomposable aCM vector bundle of rank three on $X$. If $m>4$, we have a family of such aCM vector bundles of dimension $6$.
\end{corollary}

\begin{proof}
Assume that $X$ has one component $H$ with multiplicity $1$. In this case we take as $Z$ a set of $3$ general points in $H$. Then the first assertion follows from Proposition \ref{z2}. Note that the set of all such $Z$ has dimension $6$. Now assume that $X$ has a component $H$ with multiplicity $3$. Fix a general point $p\in H$ and take a general line $L\subset \PP^3$ with $p\in L$. Then set $Z$ to be the connected component of the scheme $X\cap L$ with $p$ as its reduction. Then we may get the assertion from Proposition \ref{z2} and that $\mathrm{Pic}(X) \cong \ZZ \langle \Oo _X(1)\rangle$ by \cite[Lemma 2.5]{BHMP}. 
\end{proof}

\begin{proposition}\label{z2+0}
Let $X\subset \PP^3$ be a surface of degree $m\ge 4$ with an irreducible component $Y$ appearing with multiplicity $2$ in $X$. Fix $p\in Y_{\mathrm{reg}}$ so that $T$ is the only irreducible component of $X$ containing $p$. For a general line $L\subset \PP^3$ containing $p$, let $Z\subset X$ be the connected component of $L\cap X$ with $p$ as its reduction. We have $\deg (Z)=2$ and there is an aCM  vector bundle $\Ee_Z$ of rank two fitting into an exact sequence
\begin{equation}\label{eqz2+0}
0 \to \Oo _X(m-4)\to \Ee _Z\to \Ii _{Z,X}\to 0.
\end{equation}
The set of all isomorphism classes of $\Ee _Z$ is uniquely parametrized by a $4$-dimensional irreducible quasi-projective variety $\Delta$ satisfying the following. 
\begin{itemize}
\item [(i)] For any $\Ee _Z\in \Delta$, there are no integers $a, b$ with $\Ee _Z \cong \Oo _X(a)\oplus \Oo _X(b)$.
\item [(ii)] A very general $\Ee _Z\in \Delta$ is indecomposable.
\item [(iii)] If $\mathrm{Pic}(X) \cong \ZZ \langle \Oo _X(1)\rangle $, then each $\Ee _Z\in \Delta$ is indecomposable.
\item [(iv)] If $\ZZ \langle \Oo _X(1)\rangle $ are the only aCM line bundles on $X$, then each $\Ee _Z\in \Delta$ is indecomposable.
\end{itemize} 
\end{proposition}

\begin{proof}
Since no other component of $X$ than $Y$ contains $p$ and $p$ is a smooth point of $X$, we have $\deg (Z)=2$; it is sufficient to take as $L$ any line through $p$ not contained in the tangent plane $T_pY$ of $Y$ at $p$. 

Since $\omega _X\cong \Oo _X(m-4)$, we have $h^0(\Oo _X(4-m)\otimes \omega _X)=1$ and $\Oo _X(4-m)\otimes \omega _X$ is globally generated. Thus we have $h^0(\Ii _{p,X}\otimes \Oo _X(4-m)\otimes \omega _X)=0$. Since $Z$ is a locally complete intersection, the Cayley-Bacharach condition is satisfied for (\ref{eqz2+0}) and so there is a locally free $\Ee_Z$ fitting into (\ref{eqz2+0}); see \cite{cat}. 

Since $\Oo _X(1)$ is very ample and $\deg (Z) =2$, we get $h^1(\Ee_Z (t)) =0$ for all $t>0$ by (\ref{eqz2}). Note that $\det (\Ee_Z ) \cong \Oo _X(m-4)$ and $\Ee_Z$ is a vector bundle of rank two. This implies $\Ee_Z ^\vee \cong \Ee_Z (4-m)$. For $t<0$, we have $h^1(\Ee_Z (t)) = h^1(\Ee_Z ^\vee (m-t-4)) = h^1(\Ee_Z (-t)) =0$ by Serre's duality. Now consider the coboudary map $\delta _1: H^1(\Ii _{Z,X}) \rightarrow H^2(\Oo _X(m-4))\cong \mathbf{k}$ with $\mathrm{ker}(\delta _1) =H^1(\Ee_Z)$. The dual of $\delta_1$ is the map 
$$\mathrm{Hom}_X(\Oo_X(m-4), \Oo_X(m-4)) \to \mathrm{Ext}_X^1(\Ii_{Z,X}, \Oo_X(m-4)) $$
sending the identity map to the element corresponding to $\Ee_Z$. This implies that $\delta_1$ is non-zero and hence and $h^1(\Ee_Z)=0$. Thus $\Ee_Z$ is aCM.

The set of all $p\in Y_{\mathrm{reg}}$ such that $Y$ is the only irreducible component of $X$ containing $p$ is an irreducible $2$-dimensional variety $\Delta '$. For each $p\in \PP^3$ the set of all lines through $p$ is a $\PP^2$. Define a variety $\Delta$ as follows:
$$\Delta:=\{(p,L)~|~p\in \Delta' \text{ and }L \text{ a line in }\PP^3 \text{ with }p\in L \text{ and } L\nsubseteq T_pY\}.$$
Since $m\ge 4$, we have $h^0(\Ii _{Z,X}(4-m)) =0$. Thus (\ref{eqz2+0}) gives $h^0(\Ee _Z(4-m)) =1$. Thus the isomorphism classes of $\Ee _Z$ uniquely determines $Z$, i.e. if $\Ee _Z\ncong \Ee _{Z'}$, then we get $Z\ne Z'$. For two elements $(p_1, L_1), (p_2, L_2)\in \Delta$, let $Z_i$ be the subscheme of degree $2$ determined by $(p_i, L_i)$ for each $i=1,2$. Since each $p_i$ is the reduction of $Z_i$ and $L_i$ is the line spanned by $Z_i$, the variety $\Delta$ uniquely parametrizes the isomorphism classes of the aCM vector bundles $\Ee_Z$.

Assume $\Ee_Z \cong \Oo _X(a)\oplus \Oo _X(b)$ for some integers $a, b$ with $a\ge b$. Since $\det (\Ee_Z ) \cong \Oo _X(m-4)$, we have $b= m-4-a$. But since $h^0(\Ee_Z (4-m)) =1$, the only possibility is that $a=4-m$ and $b<0$, a contradiction. Thus we get (i). We may get (ii) as in the proof of Theorem \ref{e2}. Now assume that $\Ee_Z$ is decomposable, say $\Ee_Z \cong \Aa_1 \oplus \Aa_2$ with each $\Aa_i$ a line bundle. Since $\Ee_Z$ is aCM, each $\Aa_i$ is also aCM. Thus (iii) and (iv) follow from (i).
\end{proof}

\begin{remark}
In case $m=2$, i.e. $X=2H$ the double plane with a hyperplane $H\subset \PP^3$, the vector bundle $\Ee _Z$ described in Proposition \ref{z2+0} is the vector bundle $\Oo _X(-1)^{\oplus 2}$.
\end{remark}

\begin{theorem}\label{i1}
Let $X\subset \PP^3$ be a surface of degree $m\ge 4$ with $X_{\mathrm{reg}} \ne \emptyset$, i.e. $X$ has an irreducible component $Y$ appearing with multiplicity $1$.  We further assume that either $\mathrm{Pic}(X) =\ZZ\langle \Oo _X(1)\rangle $ or $X$ is integral. For a fixed integer $s>0$ and a set $S\subset X_{\mathrm{reg}}\cap Y$ with $\sharp (S)=s$, a general sheaf $\Ee _S$ fitting into an exact sequence
\begin{equation}\label{eqi1}
0 \to \Oo _X(m-3)^{\oplus s} \stackrel{v}{\to} \Ee _S \to \oplus _{p\in S}\Ii_{p,X}\to 0, 
\end{equation}
is a locally free, indecomposable and aCM sheaf of rank $2s$. Moreover, if $S'\subset X_{\mathrm{reg}}\cap Y$ is another set with $\sharp (S')=s$ and $S'\ne S$, then we have $\Ee _{S'}\ncong \Ee _S$.
\end{theorem}


We have $\mathrm{ext}_X^1(\Ii_{p,X}, \Oo_X(m-3))=h^1(\Ii_{p,X}(-1))=1$ for each $p\in X_{\mathrm{reg}}$ by Serre's duality. So the extension $\Ee_S$ corresponds to an element in a finite dimensional vector space
$$\mathbb{E}(S):=\mathrm{Ext}_X^1(\oplus_{p\in S}\Ii_{p,X}, \Oo_X(m-3)^{\oplus s})\cong \mathbf{k}^{s^2}.$$
If $s=1$, say $S=\{p\}$, the dimension of $\mathbb{E}$ is one. Thus there exists a unique non-trivial extension. Denote this non-trivial extension simply by $\Ee_p$. 

In Theorem \ref{i1}, a ``general'' choice of $\Ee_S$ means that there exists a non-empty Zariski open subset $\mathbb{U}\subset \mathbb{E}(S)$ such that the middle term of any extension in $\mathbb{U}$ is aCM, locally free and indecomposable.

\begin{proof}[Proof of Theorem \ref{thth}:]
The family $\Sigma$ of all $S\subset X_{\mathrm{reg}}$ with $\sharp (S) =s$ clearly has dimension $2s$. By Theorem \ref{i1}, if $S$ and $S'$ are two distinct sets in $\Sigma$, then we get $\Ee_{S} \ncong \Ee_{S'}$. Now there is a universal family on any $\mathrm{Ext}^1$-group of families of sheaves with $\Sigma \times X$ as its base. Thus, we get a family of aCM locally free and indecomposable vector bundles with as a parameter space a rank $s^2$ vector bundle over $\Sigma$; the fibre of this vector bundle over $S\in \Sigma$ is $\mathbb{E}(S)$, corresponding to $S$. Taking a non-empty open subset $V$ of $\Sigma$ on which this vector bundle is trivial we get a family of pairwise non-isomorphic sheaves, at least if we restrict $V$, so that all sheaves in the family are locally free, aCM and indecomposable.
\end{proof}

\begin{remark}
For a surface $X$ as in Theorem \ref{i1} and Theorem \ref{thth}, the algebraic group $\mathrm{Aut}(X)$ has finite dimension; it is often zero-dimensional. Hence there exists an integer $t_0$ such that for every even integer $r$, $X$ has a family of dimension at least $r-t_0$, consisting of indecomposable aCM vector bundles of rank $r$ on $X$, such that for any two distinct elements $\Ee$, $\Ee'$ in the family there is no $f\in \mathrm{Aut}(X)$ with $f^\ast (\Ee )\cong \Ee'$.
\end{remark}


\section{Proof of Theorem \ref{i1}}\label{sec4}

Set $\mathbb{E}'(S)$ to be the set of all elements in $\mathbb{E}(S)$ whose corresponding middle term is locally free and aCM.

\begin{lemma}\label{x01}
$\mathbb{E}'(S)$ is a non-empty open subset of $\mathbb{E}(S)$.
\end{lemma}

\begin{proof}
Since being locally free and aCM are both open properties in a flat family, $\mathbb{E}'(S)$ is an open subset of $\mathbb{E}(S)$. Thus it is sufficient to prove that $\mathbb{E}'(S)  \ne \emptyset$. Proposition \ref{a1.1} gives the case $s=1$. For $s>1$, we may find a direct sum of aCM vector bundles of rank two fitting into (\ref{eqi1}), i.e. take $\oplus _{p\in S} \Ee_p$. This implies $\mathbb{E}'(S) \ne \emptyset$.
\end{proof}

\begin{remark}\label{x00} 
In the set-up of (\ref{eqi1}) set $\Aa := v(\Oo _X(m-3)^{\oplus s})$. By Lemma \ref{c2} and Remark \ref{c3} together with the assumption $m\ge 3$, we see that $\Aa$ is the first term of the HN filtration of $\Ee_S$. Thus we get $f(\Aa)\subseteq \Aa$ for any $f\in \End (\Ee_S )$.
\end{remark}

\begin{lemma}\label{x0}
If $\Ee$ is the middle term of an extension $\epsilon \in \mathbb{E}'(S)$, then $\Ee$ has no line bundle as a factor.
\end{lemma}

\begin{proof}
Assume that $\Ll$ is a line bundle that is a factor of $\Ee$, i.e. $\Ee = \Ll \oplus \Gg$ for some aCM vector bundle $\Gg$ of rank $2s-1$. Since $m\ge 3$, we have 
$$h^0(\Ll (3-m)) +h^0(\Gg (3-m)) =h^0(\Ee(3-m))=s.$$ 
First assume $h^0(\Ll (3-m)) =0$ and $h^0(\Gg (3-m)) =s$. Then we have $v(\Oo_X(m-3)^{\oplus s}) \subset \{0\}\oplus \Gg$ in (\ref{eqi1}) and so $\Ll \cong \Ii_{p,X}$ for some $p\in S$, a contradiction. Thus we have $h^0(\Ll (3-m)) >0$ and so $h^0(\Gg (3-m)) <s$. In particular, there is a nonzero map $u: \Oo _X(m-3) \rightarrow \Ll$. Assume for the moment that $\mathrm{Pic}(X) \cong \ZZ\langle \Oo _X(1)\rangle$ and write $\Ll \cong \Oo _X(a)$ for some $a\in \ZZ$. The map $u$ gives $a\ge m-3$. Since $m\ge 3$, (\ref{eqi1}) is the HN-filtration of $\Ee$ and we get $a=m-3$. Thus $\Gg$ fits into an exact sequence
$$0\to \Oo_X(m-3)^{\oplus (s-1)} \to \Gg \to \oplus _{p\in S}\Ii _{p,X}\to 0. $$ 
Then we get $h^1(\Gg(-1))\ge 1$ from $h^1(\Ii_{p,X}(-1))=1$ and $h^2(\Oo_X(m-4))=1$. Thus $\Gg$ is not aCM, a contradiction. If $X$ is integral, then every line bundle is stable and so (\ref{eqi1}) is the HN-filtration of $\Ee$, we get either $\Ll \cong \Oo _X(m-3)$; we get a contradiction as above, or $\Ll $ is a factor of $\oplus _{p\in S} \Ii _{p,X}$, which is not locally free, a contradiction.
\end{proof}



Let $\FF (S)$ (resp. $\FF '(S)$) be the set of isomorphism classes of middle terms of extensions in $\EE (S)$ (resp. $\EE '(S)$). Let us denote by $\Ee=\Ee(\epsilon)$ the middle term of the extension corresponding to $\epsilon \in \EE'(S)$. 

\begin{lemma}\label{x2}
For two non-empty finite sets $S_1, S_2\subset X_{\mathrm{reg}}$ with $\sharp (S_i)=s_i$, take $\Ee_i\in \FF'(S_i)$ and call $\Aa_i$ the subsheaf of $\Ee_i$ isomorphic to $\Oo_X(m-3)^{\oplus s_i}$ for each $i=1,2$. If there exists a map $f: \Ee_1 \rightarrow \Ee_2$ with $f(\Ee_1) \subset \Aa_2$, then we have $S_1 \cap S_2\ne \emptyset$. 
\end{lemma}

\begin{proof}
Since $\Hom_X (\Oo _X(m-3),\Ii _{p,X}) =0$ for all $p\in X$, we have $f(\Aa _1)\subseteq \Aa _2$. In particular, $f$ induces a nonzero map $\tilde{f}: \oplus _{p\in S_1}\Ii _{p,X} \rightarrow \oplus _{q\in S_2} \Ii _{q,X}$. This implies that $S_1\cap S_2 \ne \emptyset$.
\end{proof}

\begin{lemma}\label{x1}
Assume that $\Ee\in \FF'(S)$ is decomposable; $\Ee \cong  \Ee _1\oplus \cdots \oplus \Ee _h$ with each $\Ee _i$ indecomposable. Then there is a partition $S = \sqcup _{i=1}^{h} S_i$ with $\Ee _i\in \FF '(S_i)$ for each $i$. If there is another decomposition $\Ee \cong \Ee _1'\oplus \cdots \oplus \Ee _k'$ with each $\Ee_j'$ indecomposable, then we get $k=h$ and there is a permutation $\sigma : \{1,\dots ,h\}\rightarrow\{1,\dots ,h\}$ such that $\Ee _{\sigma (i)}' \cong \Ee _i$ for all $i$ and $\Ee _{\sigma (i)}'\in \FF (S_{\sigma
(i)})$.
\end{lemma}

\begin{proof}
We use induction on $s$. The case $s=1$ is true, because each $\Ee _p$ for $p\in X_{\mathrm{reg}}$ is indecomposable by Proposition \ref{a1.1}. Since $\Ee$ is aCM by the definition of $\FF (S)$, each $\Ee _i$ is also aCM. We consider the subsheaf $\Aa \cong \Oo_X(m-3)^{\oplus s}\subset \Ee$ as in Remark \ref{x00} and set $\Gg _i:= \Aa \cap \Ee_i$. Since the HN filtration of $\Ee$ is obtained from the ones of each factors, we have 
$$\Aa \cong \oplus _{i=1}^{h} \Gg _i \phantom{AA}\text{and}\phantom{AA}\oplus _{p\in S}\Ii_{p,X}\cong \oplus_{i=1}^h \Ee _i/\Gg _i.$$
By Lemma \ref{x0} we have $\Gg _i\subsetneq \Ee _i$ for all $i$. By Remark \ref{c3} we may write $S = \sqcup _{i=1}^{h} S_i $ with $\Ee_i/\Gg _i\cong \oplus _{p\in S_i}\Ii_{p, X}$. Since $\Ee _i/\Gg _i \ne 0$, we have $S_i\ne \emptyset$ for all $i$. Thus the set $\{S_1,\dots ,S_h\}$ gives a partition of $S$. To prove the first part of the lemma it is sufficient to prove that $\sharp (S_i) =\mathrm{rank}(\Gg _i)/2$ for all $i$. If this is not true, then there is $i\in \{1,\dots ,h\}$ with $\sharp (S_i) > \mathrm{rank}(\Gg _i)/2$, i.e. $\mathrm{rank}(\Gg _i\cap \Aa)> \sharp (S_i)$. The exact sequence
$$0 \to \Aa\cap \Gg _i\to \Gg _i\to \oplus _{p\in S_j}\Ii _{p,X}\to 0$$ 
gives $h^1(\Gg _i)\ge \sharp (S_i)-\mathrm{rank}(\Gg _i\cap\Aa) >0$. In particular, $\Gg _i$ is not aCM, a contradiction.

Now we check the last assertion of the lemma. Take two partitions 
$$S = S_1\sqcup \cdots \sqcup S_h= S_1'\sqcup \cdots \sqcup S_k'$$ 
such that there is a decomposition 
$$\Ee \cong  \Ee _1\oplus \cdots \oplus \Ee _h \cong  \Ee _1'\oplus \cdots \oplus \Ee _k'$$ 
with $\Ee_i\in \FF'(S_i) $ and $\Ee _j'\in \FF'(S_j')$ indecomposable. By the Krull-Schmidt theorem in \cite{Atiyah}, we get $h=k$ and there is a permutation $\sigma : \{1,\dots ,h\}\rightarrow \{1,\dots ,h\}$ such that $\Bb _{\sigma (i)} \cong \Ee _i$ for all $i$. By renaming $\{\Ee _1',\dots ,\Ee_h'\}$, we may assume that $\Ee _i' \cong \Ee _i$ for all $i$. This implies
$$\sharp (S_i)=\mathrm{rank}(\Ee_i)/2=\mathrm{rank}(\Ee_i')/2=\sharp(S_i').$$
Now fix an isomorphism $f_i: \Ee _i\rightarrow \Ee _i'$ for each $i$. Since (\ref{eqi1}) gives the HN filtrations of $\Ee _i$ and $\Ee_i'$, the map $f$ induces an isomorphism $\tilde{f_i}: \oplus _{p\in S_i}\Ii _{p,X}\rightarrow  \oplus _{p\in S_i'} \Ii _{p,X}$. Since $p$ is the unique point of $X$ at which $\Ii _{p,X}$ is not locally free, we get $S_i=S_i'$. For each $i$, let $\Aa_i$ be the unique subsheaf of $\Ee _i$ isomorphic to $\Oo _X(m-3)^{\sharp (S_i)}$. Then for any embedding $u: \Ee _i \rightarrow  \Ee _1\oplus \cdots \oplus \Ee _h$, the composition $v_j\circ \pi _j \circ u$
$$\Ee_i \stackrel{u}{\to} \Ee_1\oplus \cdots \oplus \Ee_h \stackrel{\pi_j}{\to} \Ee_j \stackrel{v_j}{\to} \oplus_{p\in S_j}\Ii_{p,X}$$
is zero for any $j\ne i$ by Lemma \ref{x2}, where $\pi _j: \Ee \rightarrow \Ee _j$ is the projection and $v_j: \Ee _j\rightarrow \oplus _{p\in S_j}\Ii _{p,X}$ is the surjection in (\ref{eqi1}) for $S_j$. Since $u$ is an embedding, we see that $v_i\circ \pi _i \circ u$ is surjective. Thus $\Gg:= \pi _i(u(\Ee _i))$ is a subsheaf with $v_i(\Gg ) = \oplus _{p\in S_i}\Ii _{p,X}$. 
\end{proof}

\begin{lemma}\label{ppo}
With the setting as in Theorem \ref{i1}, we have $\mathrm{ext}_X^1(\Ee_p, \Ee_q)\ge 2$ for two points $p,q\in X_{\mathrm{reg}}$, possibly $p=q$. 
\end{lemma}

\begin{proof}
Set $\Ff _o:= \Ee _o(3-m)$ for $o\in \{p,q\}$. Since $\Ext ^i_X(\Ee _p,\Ee _q) \cong \Ext_X ^i(\Ff _p,\Ff _q)$, we have $\chi (\Ee _p\otimes \Ee _q^\vee ) = \chi (\Ff _p\otimes \Ff _q^\vee)$. Since Euler's characteristic is constant in a flat family of vector bundles and $p,q\in X_{\mathrm{reg}}$, it is sufficient to compute $\chi (\Ff _p\otimes \Ff _q^\vee)$ when $X$ is smooth. Since a smooth surface in $\PP^3$ is connected, the same observation applied to a family of vector bundles on $X$ shows $\chi (\Ff _p\otimes \Ff _q^\vee) =\chi (\Ff _p\otimes \Ff _p^\vee)$.

We have an exact sequence
\begin{equation}\label{eqii1}
0 \to \Oo _X \stackrel{v}{\to} \Ff _p \stackrel{w}{\to} \Ii _{p,X}(3-m)\to 0
\end{equation}
with $\det (\Ff _p)\cong \Oo _X(3-m)$ and $c_2(\Ff _p)=1$. Since $X\subset \PP^3$ is a surface of degree $m$, we have $c_1(\Ff _p)^2 = m(m-3)^2$. By Riemann-Roch for $\mathcal{E}nd (\Ff _p)$, we have 
\begin{align*}
\chi (\mathcal{E}nd (\Ff _p ))& = c_1(\Ff _p)^2 -4c_2(\Ff _p) + 4\chi (\Oo _X) = m(m-3)^2 -4 +4\binom{m-1}{3} +4\\
&=\frac{1}{6}\left(10m^3-60m^2+98m-24\right).
\end{align*}
In particular, we have $\chi \sim \frac{5}{3}m^3$ for $m\gg 0$. Note that by Serre's duality we have $h^2(\Ff _p\otimes \Ff _p^\vee )=h^0(\Ff _p\otimes \Ff _p^\vee (m-4))$.

\quad \emph{Claim 1:} We have $\mathrm{hom}_X (\Ff _p,\Ff _p) =1+\binom{m}{3}$.

\quad \emph{Proof of Claim 1:} We have $\mathrm{hom}_X (\Ii _{p,X}(3-m),\Oo _X) =h^0(\Oo _X(m-3)) = \binom{m}{3}$ and any nonzero map $\Ii _{p,X}(3-m) \rightarrow \Oo _X$ induces an element in $\Hom_X(\Ff_p, \Ff_p)$ with rank one as the following composition:
$$\Ff_p \stackrel{w}{\to} \Ii_{p,X}(3-m) \to \Oo_X \stackrel{v}{\to} \Ff_p.$$
The vector space $\Hom_X (\Ff _p,\Ff _p)$ also contains the nonzero multiples of the identity map $\Ff _p\rightarrow \Ff _p$ and these maps have rank two. Thus we get $h^0(\Ff _p\otimes \Ff _p^\vee )\ge 1+\binom{m}{3}$. On the other hand, for any $f\in \Hom_X (\Ff _p,\Ff _p)$ we get $w\circ f \circ (v(\Oo _X)) \subseteq v(\Oo _X)$ from $h^0(\Ii _{p,X}(3-m)) =0$. Thus $w\circ f\circ v$ induces a map $f_1: \Oo _X\rightarrow \Oo_X$, which is induced by the multiplication by $c\in \mathbf{k}$. Hence $f-c{\cdot}\mathrm{Id} _{\Ff _p}$ is induced by a unique $g\in \Hom_X (\Ii _{p,X}(3-m),\Ff _p)$. Since $\Ff _p$ is locally free and $X$ is smooth at $p$, we have $\Hom_X (\Ii _{p,X}(3-m),\Ff _p) = H^0(\Ff _p(m-3))$. By
(\ref{eqii1}) we have $h^0(\Ff_p (m-3)) =\binom{m}{3}$ and so $\mathrm{hom}_X (\Ff _p,\Ff _p) \le 1+\binom{m}{3}$. \qed

 \quad \emph{Claim 2:} We have $\mathrm{hom}_X (\Ff _p,\Ff _p(m-4)) \ge \binom{2m-4}{3} +2\binom{m-1}{3} -\binom{m-4}{3}-1$.

\quad \emph{Proof of Claim 2:} For any $f \in \Hom_X (\Ff _p, \Ff _p(4-m))$, set $f_1:= f_{|v(\Oo _X)}$. Since $h^0(\Oo _X(-1)) =0$, we have $w\circ f_1 =0$ and so $f_1(v(\Oo _X)) \subset v(\Oo _X(m-4)))$. Take $f$ with $f_1\equiv 0$. Such a map $f$ is uniquely determined by an element in $\Hom_X (\Ii _{p,X}(3-m), \Ff _p(m-4))$ and the converse also holds. Since $\Ff _p(m-4)$ is locally free and $X$ is smooth at $p$, we have $\Hom_X (\Ii _{p,X}(3-m), \Ff _p(m-4)) = \Hom_X (\Oo _X(3-m),\Ff _p(m-4)) = H^0(\Ff _p(2m-7))$. Since $h^1(\Oo _X(t)) =0$ for any $t\in \ZZ$, (\ref{eqii1}) gives 
$$h^0(\Ff _p(2m-7)) = h^0(\Oo _X(2m-7)) +h^0(\Oo _X(m-4))-1 = \binom{2m-4}{3} -\binom{m-4}{3}+\binom{m-1}{3}-1.$$ 
Note that a map $f$ obtained by a composition
$$\Ff_p \stackrel{w}{\to} \Ii_{p,X}(3-m) \to \Oo_X(m-4) \stackrel{v}{\to} \Ff_p(m-4)$$
has $f_1 \equiv 0$. Now for any linear subspace $W\subset \Hom_X (\Ff _p,\Ff _p(m-4))$ such that $f_1\not\equiv 0$ for any $f\in W\setminus \{0\}$, we would get 
$$\mathrm{hom}_X (\Ff _p,\Ff _p(m-4)) \ge \binom{2m-4}{3} -\binom{m-4}{3}+\binom{m-1}{3}-1+\dim W.$$
We may choose $W$ to consist of the compositions of the identity map $\Ff _p\rightarrow \Ff _p$ with the multiplication by an element of $H^0(\Oo _X(m-4))$. Then we have $\dim W=\binom{m-1}{3}$.   \qed

Combining Claims 1 and 2, we get
\begin{align*}
h^0(\Ff _p\otimes \Ff _p^\vee) +h^2(\Ff _p\otimes \Ff _p^\vee )&\ge \binom{2m-4}{3} +\binom{m}{3}+2\binom{m-1}{3}-\binom{m-4}{3}\\
&=\frac{1}{6}\left( 10m^3-60m^2+98m-12\right).
\end{align*}
Thus we have 
$$
h^1(\Ff_p\otimes \Ff_p^\vee)=h^0(\Ff_p\otimes \Ff_p^\vee)+h^2(\Ff_p\otimes \Ff_p^\vee)-\chi(\mathcal{E}nd(\Ff_p))\ge 2
$$
and so we get the assertion. 
\end{proof}

\begin{proof}[Proof of Theorem \ref{i1}:]
By Remark \ref{x00} (\ref{eqi1}) is the HN filtration of $\Ee_S$. Proposition \ref{a1.1} gives the case $s=1$. For $s>1$, we may find a direct sum of $s$ vector bundles of rank $2$ from the case $s=1$, fitting into (\ref{eqi1}): just take $\oplus _{p\in S} \Ee_p$. So a general extension in $\mathbb{E}(S)$ has a locally free and aCM middle term, because being local free and aCM are both open conditions. 

Note that $h^0(\Ee_S(3-m))=s$ from (\ref{eqi1}). In particular there is a unique subsheaf $\Aa\subset \Ee_S$ isomorphic to $\Oo _X(m-3)^{\oplus s}$ and for each $f\in \mathrm{Hom}(\Oo _X(m-3),\Ee_S)$ we have $f(\Oo _X(m-3)) \subseteq \Aa$. Now by Lemma \ref{c2} and Remark \ref{c3}, the extension (\ref{eqi1}) is the HN filtration of $\Ee_S$. By uniqueness of the HN filtration, we get $\Ee_S \ncong \Ee_{S'}$ for $S\ne S'$. 

Now it remains to show the indecomposability of $\Ee_S$. By Lemma \ref{x0}, there is no rank one factor of $\Ee_S$.

\quad \emph{Claim 1:} For two distinct points $p, q$ in $X_{\mathrm{reg}}$, we have 
$$\mathrm{Hom}_X(\Ii _{p,X},\Ii _{q,X}) =0, \mathrm{Hom}_X(\Ee_p,\Ii _{q,X}) =0 \text{ and } \mathrm{Ext}_X^1(\Ii_{p,X} , \Ii_{q,X})=0.$$  

\quad \emph{Proof of Claim 1:} By an extension theorem for locally free sheaves in \cite[Exercise I.3.20]{Hartshorne}, we have $\mathrm{Hom}_X(\Ii _{p,X},\Ii _{q,X}) = \mathrm{Hom}_X(\Oo _X,\Ii _{q,X}) =0$. The second vanishing is obtained from the first vanishing and $\mathrm{Hom}_X(\Oo _X(m-3),\Ii _{q,X}) =0$. For the last vanishing, we apply the functor $\mathrm{Hom}_X(\Ii_{p,X}, -)$ to the standard exact sequence for $\Ii_{q,X}\subset \Oo_X$ and obtain an exact sequence
$$0 \to \mathrm{Hom}_X(\Ii_{p,X}, \Oo_X) \to \mathrm{Hom}_X(\Ii_{p,X}, \Oo_q) \to \mathrm{Ext}_X^1(\Ii_{p,X}, \Ii_{q,X}) \to \mathrm{Ext}_X^1(\Ii_{p,X}, \Oo_X)$$
by the first vanishing in the Claim. Here we have
$$\mathrm{Hom}_X(\Ii_{p,X}, \Oo_X) \cong \mathrm{Hom}_X(\Ii_{p,X}, \Oo_q) \cong \mathbf{k}$$
and $\mathrm{Ext}_X^1(\Ii_{p,X}, \Oo_X) \cong H^1(\Ii_{p,X}(m-4))^\vee$ by Serre's duality. Then we get the assertion from the assumption that $m\ge 4$.  \qed

\quad {(a)} First assume $s=2$ and take  two distinct points $p, q$ in $X_{\mathrm{reg}}$.

\quad \emph{Claim 2:} If there exists a sheaf $\Gg\ncong \Ee _p\oplus \Ee _q$ fitting into the exact sequence
\begin{equation}\label{eqa3.1}
0 \to \Ee _p\stackrel{u}{\to} \Gg \stackrel{v}{\to} \Ee _q\to 0,
\end{equation}
then the case $s=2$ is true.

\quad \emph{Proof of Claim 2:} Such a sheaf $\Gg$ would be locally free and aCM with rank $4$. Since $h^1(\Oo _X)=0$ and (\ref{eqi1}) gives the HN filtrations of $\Ee _p$ and $\Ee _q$ by Lemmas \ref{c2}
and Remark \ref{c3}, $\Gg$ has a subsheaf $\Ff \cong \Oo _{X}(m-3)^{\oplus 2}$ such that $\Gg /\Ff$ is an extension of $\Ii _{q,X}(1)$ by $\Ii _{p,X}(1)$. Claim 1 gives $\Gg/\Ff \cong \Ii_{p,X}\oplus \Ii_{q,X}$ and so we get $\Gg \cong \Ee _S$ with $S =\{p,q\}$.\qed

\quad \emph{Claim 3:} If $\Gg \cong \Ee _p\oplus \Ee _q$ for all $\Gg$ in (\ref{eqa3.1}), then we have $\Ext_X ^1(\Ee _q,\Ee _p)=0$.

\quad \emph{Proof of Claim 3:} Let $\Gg \cong \Ee _p\oplus \Ee _q$ fitting into (\ref{eqa3.1}) correspond to $\epsilon \in \Ext_X ^1(\Ee _q,\Ee _p)$. Then it is sufficient to prove that $\epsilon =0$, or $\mathrm{ker}(v) \cong \Ee_p\oplus \{0\}$. But since $\mathrm{ker}(v) \cong \Ee _p$, it is sufficient to prove that either $\Ee _p\oplus \{0\} \supseteq \mathrm{ker}(v)$ or $\Ee _p\oplus \{0\} \subseteq \mathrm{ker}(v)$. Assume $v(\Ee _p\oplus \{0\}) \ne 0$. Since $\mathrm{Hom}_X(\Ee_p,\Ii _{q,X}) =0$ by Claim 1, we have $v(\Ee _p\oplus \{0\})\subseteq \Oo _X(m-3)$. This implies that the restriction of the surjection $\Ee _q\rightarrow \Ii _{q,X}$ to $v(\{0\}\oplus \Ee _q)$ is surjective. Since $h^0(\Oo _X)=1$ and $\Hom_X (\Oo _X(m-3),\Ii _{q,X}) =0$, we get either $v(\{0\}\oplus \Oo _X(m-3)) =0$ or $v$ induces an isomorphism $\{0\}\oplus \Oo _X(m-3)\rightarrow \Oo _X(m-3)$. Assume for the moment $v(\{0\}\oplus \Oo _X(m-3)) =0$. Since $v(\Ee _p\oplus \{0\})$ maps to $0$ in $\Ii _{q,X}$, we get that $v(\{0\}\oplus \Ee _q)$ is a subsheaf of $\Ee _q$ which maps isomorphically onto $\Ii _{q,X}$. So we get $\Ee _q\cong \Oo _X(m-3)\oplus \Ii _{q,X}$, a contradiction. Now assume $v(\{0\}\oplus \Oo _X(m-3)) =\Oo _X(m-3)$. Since $v(\{0\}\oplus \Ee _q)$ maps surjectively onto $\Ii _{q,X}$, the surjection $v$ induces an isomorphism $\{0\}\oplus \Ee _q\rightarrow \Ee _q$. Hence we get $\Ee _p\oplus \{0\} \subseteq \mathrm{ker}(v)$. \qed

\noindent Since $\Ext_X ^1(\Ee_q,\Ee _p) \ne 0$ by Lemma \ref{ppo}, Claim 3 concludes the proof of the case $s=2$.

\quad {(b)} Assume $s>2$ and that Theorem \ref{i1} holds for smaller numbers. On $\mathbb{E}(S)$ there is a universal family of extensions, i.e. a coherent sheaf $\Vv$ over $\mathbb{E}(S)\times X$ such that
for each $\epsilon \in \mathbb{E}(S)$ the sheaf $\Vv _{|\{\epsilon \}\times X}$ is the middle term $\Ee (\epsilon )$ of the extension corresponding to $\epsilon$; in general, if we take $\PP(\mathbb{E}(S))$ as a parameter space, then no such a universal sheaf exists. We call $\Vv '$ the restriction of of $\Vv$ to $\EE '(S)\times X$; we thus consider the family of aCM vector bundles induced from the extensions in $\EE'(S)$. 

Define a set $\Gamma(S)$ as follows:
$$\Gamma(S) :=\left\{(\epsilon ,\phi)~|~\epsilon \in \EE' (S) \text{ and } \phi \in \End (\Ee (\epsilon )) \text{ with }\phi ^2 = \phi \right \}.$$ 
Note that $\phi$ is a projection of $\Ee (\epsilon)$ onto a factor of $\Ee (\epsilon)$, with the exception when $\phi = \mathrm{Id}_{\Ee (\epsilon)}$ or $\phi \equiv 0$; if $\Ee (\epsilon )$ is indecomposable, only $(\epsilon ,\mathrm{Id}_{\Ee (\epsilon)})$ and $(\epsilon ,0)$ are contained in $\Gamma(S)$. Indeed, for any vector bundle $\Gg$, there exists a one-to-one correspondence:
$$\{ \phi \in \End(\Gg) ~|~ \phi^2=\phi\} \leftrightarrow \{ \text{factors of }\Gg\}$$ 
via $\phi \mapsto \mathrm{Im}(\phi) = \mathrm{ker}(\mathrm{Id}_{\Gg} -\phi)$, with $\Gg$ being associated to $\mathrm{Id}_{\Gg}$ and $0$ associated to the zero map. Thus $\Gg$ is decomposable if and only if $\End (\Gg )$ has a non-trivial idempotent. Note that $\Gamma(S)$ is a closed in the total space of the vector bundle $\mathcal{H}om(\Vv' , \Vv')$ over $\EE '(S)\times X$. By Lemma \ref{x1}, for each $\Ee (\epsilon )$ there is a unique partition of $S$ associated to any decomposition of $\Ee (\epsilon )$ with only finitely many indecomposable factors by the Krull-Schmidt theorem in \cite{Atiyah}. By Lemma \ref{x1} for each $\Ee \in \FF' (S)$ each isomorphism class of factors of $\Ee$ corresponds to a unique subset of $S$; $\Ee$ and $0$ correspond to $S$ and $\emptyset$, respectively. For each $(\epsilon,\phi)\in \Gamma(S)$, let $S(\phi)$ be the subset of $S$ associated to $\mathrm{Im}(\phi)$ by Lemma \ref{x1}. Set 
$$\Gamma _0(S) :=\left\{(\epsilon, \phi)\in \Gamma(S) ~\big|~ \phi\ne 0 \text{ and } \phi\ne \mathrm{Id}_{|\Ee(\epsilon)}\right\}. $$
The goal is to show that $\Gamma_0(S)$ is not dominant over $\FF(S)$ for a general $S$. 

Note that up to now we did not use that $S$ is contained in the same connected component $Y\cap X_{\mathrm{reg}}$ of $X_{\mathrm{reg}}$. In particular the case $s=2$ holds even if $X$ has more than one irreducible components with multiplicity one and the two points of $S$ belong to different connected components of $X_{\mathrm{reg}}$. 

Now we use a monodromy argument, which requires that $S$ is contained in a connected component of $T:=X_{\mathrm{reg}}\cap Y$ and that $S$ is general in $Y$. Set $S=\{p_1, \ldots, p_s\}$ and fix an ordering of the points in $S$, along which we get an ordering of the indecomposable factors of the sheaf $\oplus _{p\in S} \Ii_{p,X}$. Together with the usual ordering on the factors of $\mathcal {O}_X(m-3)^{\oplus s}$, we may see any $\epsilon \in \mathbb{E}(S)$ as an $(s\times s)$-square matrix, say $\epsilon =(\epsilon_{ij})$ with $1\le i, j\le s$, where $\epsilon _{ij}$ is an element of the $1$-dimensional vector space $\Ext_X ^1(\Ii _{p_j,X},\Oo_X(m-3))$. Note that for a fixed integer $j$, each $\epsilon_{ij}$ with $i=1,\dots ,s$, is an element of the same $1$-dimensional vector space. We write $\Oo _X(m-3)^{\oplus s} = \CC ^s\otimes \Oo _X(m-3)$. 

\quad \emph{Claim 4:} $\Ee=\Ee(\epsilon)$ has two indecomposable factors, one of them being $\mathrm{Im}(\phi)$ and the other one being $\mathrm{ker}(\phi)$.

\quad \emph{Proof of Claim 4:} Since $\phi^2=\phi$, we have $\Ee \cong \Ff_1 \oplus \Ff_2$ with $\Ff _1:= \mathrm{Im}(\phi)$ and $\Ff_2= \mathrm{ker}(\phi)$. By the definition of $A$, we get an exact sequence
\begin{equation}\label{eert1}
0\to \Oo_X(m-3)^{\oplus k} \to \Ff_1 \to \oplus_{p\in A}\Ii_{p,X} \to 0,
\end{equation}
with $k:= \sharp (A)$. Since neither $\phi \equiv 0$ nor $\phi =\mathrm{Id} _{\Ee}$, we have $0<k<s$. Then by Lemma \ref{x1} we get an exact sequence
\begin{equation}\label{eert2}
0\to \Oo_X(m-3)^{\oplus (s-k)} \to \Ff_2 \to \oplus_{p\in S\setminus A}\Ii_{p,X} \to 0.
\end{equation}
Now we need to prove that each $\Ff_i$ is indecomposable. By the inductive assumption it is sufficient to prove that $\Ff_1$ and $\Ff_2$ are the middle terms of general extensions (\ref{eert1}) and (\ref{eert2}), respectively. Since (\ref{eqi1}) gives the HN filtration of each $\Ff_i$, there are linear subspaces $V_1,V_2\subset \CC^s$ such that $\dim V_1 =k$, $\dim V_2 =s-k$ and 
$$v(\CC ^s\otimes \Oo _X(m-3))\cap \Ff_i = V_i\otimes \Oo _X(m-3)$$
for each $i$. From $\Ee \cong \Ff_1\oplus \Ff_2$ we see that $\CC^{s} =V_1\oplus V_2$. Now we reorder  the points in $S$ so that all points of $A$ are smaller than any points of $S\setminus A$. Then $\epsilon$ can be understood as an $(s\times s)$-square matrix in a block form:
\[
\epsilon=
\left[
\begin{array}{c|c}
B_{11} & B_{12}\\
\hline
B_{21} & B_{22}
\end{array}
\right]
\]
Here the $(k\times k)$-matrix $B_{11}$ in the upper left corner, is associated to the extension (\ref{eert1}) and similarly the $((s-k)\times (s-k))$-matrix $B_{22}$ in the lower right corner, is associated to the extension (\ref{eert2}). The matrix of $\epsilon$ also has a $(k\times (s-k))$-submatrix $B_{12}$ and an $((s-k)\times k)$-submatrix $B_{21}$. Since $\epsilon$ is general, all the entries in each $B_{ij}$ are also general. In particular, $B_{11}$ and $B_{22}$ are general and this implies that each  $\Ff_i$ is general. The inductive assumption gives that each $\Ff_i$ is indecomposable.\qed

Assume that a general $\Ee=\Ee(\epsilon)$ has two indecomposable factor, i.e. the set $\Gamma_0(S)$ is dominant over $\FF(S)$. Let $\Gamma'(S)$ be an irreducible component of $\Gamma_0(S)$ dominant over $\FF (S)$ and set $A:= S(\phi)$, where $(\epsilon ,\phi )$ is any element of $\Gamma'(S)$. Now assume that $(\epsilon ,\phi)$ is general in $\Gamma'(S)$ and set $\Ee := \Ee (\epsilon)$. Note that the subset $A\subset S$ is invariant as $(\epsilon, \phi)$ varies in $\Gamma_0(S)$, due to the irreducibility of $\Gamma_0(S)$. Below we find a contradiction under the assumptions that $\Ee$ is decomposable and that $S$ is general in $\mathrm{Sym}^s(T)$. 

Let $\widetilde{\Gamma}$ be the set of all triples $(S,\Ee ,\phi)$ with $S\in \mathrm{Sym} ^s(T)$ and $(\Ee ,\phi )\in \Gamma_0 (S)$. Then $\widetilde{\Gamma}$ is an algebraic subset whose fibre over $S\in \mathrm{Sym}^s(T)$ is $\Gamma_0(S)$, with a projection map $u: \widetilde{\Gamma} \rightarrow  \mathrm{Sym} ^s(T)$. If $u$ is not dominant, then it would imply that there exists a $2s$-dimensional family of pairwise not isomorphic indecomposable aCM vector bundles of rank $2s$ on $X$. Thus we may assume that $u$ is dominant. We fix a general $S\in \mathrm{Sym} ^s(T)$ and fix an irreducible component $\Gamma '(S)$ of $\Gamma (S)$ to which we apply the previous construction with the partition $A\sqcup (S\setminus A)$ of $S$ attached to $\Gamma'(S)$. Let $\widetilde{\Gamma}'$ be any irreducible component of $\widetilde{\Gamma}$ containing $\Gamma '(S)$ such that $u_{|\widetilde{\Gamma} '}$ is dominant.

Let $\Vv$ denote a non-empty Zariski open subset of $\mathrm{Sym}^s(T)$ containing $S$ such that for every $T\in \Vv$ a general $\Ee_T \in \EE (T)$ has exactly two indecomposable factors, one associated to a subset $F$ of $T$ with $|F| = |A|=k$ and the other one associated to $T\setminus E$. Now we fix $p\in A$ and $q\in S\setminus A$. Since $Y_{\mathrm{reg}}$ is a connected manifold and $p, q\in Y_{\mathrm{reg}}$, there exists a connected smooth affine curve $U\subset \AA^1(\mathbf{k})$ with a map $\phi :  U\rightarrow  Y_{\mathrm{reg}}$ such that $\phi (t_0) = p$ and $\phi (t_1) =q$ for some $t_0, t_1 \in U$, and $\phi(U)$ passes no other points of $S$. Similarly we may consider a map $\phi' : U \rightarrow Y_{\mathrm{reg}}$ with $\phi'(t_1)=p$ and $\phi'(t_0)=q$ such that $\phi(t) \ne \phi'(t)$ for any $t\in U$. For each $t\in U$, set 
$$A_t:= (A\setminus \{p\})\cup \{\phi(t)\} \phantom{A}, \phantom{A}S_t:= (S\setminus \{p,q\}) \cup \{\phi(t) , \phi'(t)\},$$
e.g. $(A_{t_0}, S_{t_0})=(A_{t_1}, S_{t_1})=(A,S)$. Restricting $U$ to an open neighborhood of $\{t_0, t_1\}$, we may assume that $S_t\in \Vv$ for all $t\in U$. Then for each $t\in U$ we have a partition $S_t = A_t\sqcup (S_t\setminus A_t)$ such that a general $\Ee _{S_t}\in \Gamma' (S_t)$ has exactly two indecomposable factors, one associated to $A_t$ and the other associated to $S_t\setminus A_t$, due to the choice of $\widetilde{\Gamma}'$. 

We start from $t=t_0$ and vary $t$ in $U$ to arrive at $t=t_1$, where we have $S_{t_1} =S=A_q\sqcup (S\setminus A_q)$ with $A_q=(A\setminus \{p\})\cup \{q\}$. Since $s>2$, we have $\{A,S\setminus A\} \ne \{A_q,S\setminus A_q\}$, contradicting the assumption that $\Ee _S$ has exactly two indecomposable factors.
\end{proof}

\section{Non-locally free aCM sheaf}\label{sec5}
In this section, we let $X\subset \PP^N$ be a closed subscheme with pure dimension $n$ at least two.  Assume that each local ring $\Oo _{X,x}$ with $x\in X$, has depth $n$ and that $X$ is aCM with respect
to $\Oo _X(1)$, i.e. $h^i(\Ii _{X,\PP^N}(t)) =0$ for all $t\in \ZZ$ and all $1\le i\le n-1$. The exact sequence
$$0\to \Ii _{X,\PP^N}(t) \to \Oo _{\PP^N}(t) \to \Oo _X(t)\to 0$$
shows that $h^i(\Ii _{X,\PP^N}(t)) = h^{i-1}(\Oo _X(t))$ for all $i\ge 2$. Hence we may restate our assumption as $h^1(\Ii _{X,\PP^N}(t)) =0$ and $h^i(\Oo _X(t)) =0$ for all $t\in \ZZ$ and $i=1,\dots ,n-2$. By a theorem of Serre, the condition that $h^i(\Oo _X(-x)) =0$ for $x\gg 0$ and $i=1,\dots ,n-2$, plus having positive depth at each $x\in X$, is equivalent to all $\Oo _{X,x}$ having depth $n$. Since $h^1(\Ii _{X,\PP^N})=0$, we have $h^0(\Oo _X)=1$ and in particular $X$ is connected. Since $h^1(\Ii _{X,\PP^N}(1)) =0$, $X$ is linearly normal in the linear subspace of $\PP^N$ spanned by $X$. Since $n\ge 2$ we have $h^1(\Oo _X)=0$ an so $\mathrm{Pic}(X)$ is a finitely generated abelian group. 

Fix an irreducible component $Y$ of $X_{\mathrm{red}}$. If $X$ is a hypersurface in $\PP^N$, then the multiplicity $\mu \ge 1$ is well-defined. In the general case we do not need the notion of the multiplicity $\mu$ of $Y$ in $X$ at a general point of $Y$. In this section we need knowledge only on whether $\mu =1$ or $\mu >1$. We say that $Y$ has multiplicity $\mu=1$ if $X$ is reduced at a general $x\in Y$, i.e. there is a non-empty open subset $U\subseteq Y$ such that $\Oo_{X,x}=\Oo_{Y,x}$ for all $x\in U$. Otherwise we say that $Y$ has multiplicity $\mu>1$. We are interested only in the case $X$ not integral; if $Y$ has multiplicity $1$, then we have other irreducible components of $X_{\mathrm{red}}$.

\begin{lemma}\label{zz1}
Let $C\subset X$ be a reduced aCM subvariety of pure dimension $n-1$. Then its ideal sheaf $\Ii _{C,X}$ is an aCM $\Oo _X$-sheaf such that
\begin{itemize}
\item it is locally free outside $C$ and 
\item for any closed subscheme $Y\subsetneq X$, it is not an $\Oo _Y$-sheaf. 
\end{itemize}
\end{lemma}

\begin{proof}
Since $C$ is aCM as a closed subscheme of $\PP^N$ and $C$ has pure dimension $n-1$, we have $h^1(\Ii _{C,\PP^N}(t)) =0$ for all $t\in \ZZ$. Thus the restriction map $\rho _t: H^0(\Oo _{\PP^N}(t)) \rightarrow H^0(\Oo _C(t))$ is surjective for any $t\in \ZZ$. Since $\rho _t$ factors through the restriction map $\eta _t: H^0(\Oo _X(t)) \rightarrow H^0(\Oo _C(t))$, $\eta _t$ is surjective. Since $\eta _t$ is surjective and $h^1(\Oo _X(t))=0$, we have $h^1(\Ii _{C,X}(t)) =0$. This implies that $\Ii _{C,X}$ is aCM. From $\Ii _{C,X\setminus C} \cong \Oo _{X\setminus C}$, we see that $\Ii _{C,X}$ is locally free and of rank $1$ outside $C$. Since $C$ is not an irreducible component of $X_{\mathrm{red}}$ and $\Ii _{C,X}$ is locally free of positive rank outside $C$, there is no closed subscheme $Y\subsetneq X$ with $\Ii _{C,X}$ an $\Oo _Y$-sheaf.
\end{proof}

\begin{proposition}\label{zz2}
Fix an irreducible component $Y$ of $X_{\mathrm{red}}$. For a fixed integer $e>0$ and any integral divisor $C\in |\Oo_Y(e)|$, define
$$\Sigma _C:=\left\{ p\in Y~|~ \Ii_{C,X} \text{ is not locally free at }p\right\}.$$
\begin{itemize}
\item [(i)] If $Y$ has multiplicity $\mu>1$ in $X$, then we have $\Sigma _C = C$, i.e. for all $p\in C$ the sheaf $\Ii _{C,X}$ is not locally free at $p$. For any two integral curves $C_1,C_2\in |\Oo_Y(e)|$, we have $\Ii _{C_1,X}\cong \Ii _{C_2,X}$ if and only if $C_1=C_2$.
\item [(ii)] Assume that $Y$ has multiplicity $\mu=1$ and that $X$ is not integral. Let $F\in |\Oo _Y(m-1)|$ be the complete intersection of $Y$ with the other components of $X$, counting multiplicities. If $F\ne \emptyset$, then $F$ has pure dimension $n-1$ and $F\cap C \ne \emptyset$ with $\Sigma _C =(F\cap C)_{\mathrm{red}}$.
\item [(iii)] For any two integral divisors $C_1,C_2\in |\Oo_Y(e)|$ such that $\Ii _{C_1,X}\cong \Ii _{C_2,X}$, we have $\Sigma_{C_1}=\Sigma_{C_2}$; in case (i) we have the converse.
\end{itemize}
\end{proposition}

\begin{proof}
By Lemma \ref{zz1} the sheaf $\Ii _{C,X}$ is aCM and locally free with rank $1$ at all $p\in X\setminus C$.  Fix $p\in C$ and assume that $\Ii _{C,X}$ is locally free at $p$. Then there is $w\in (\Ii _{C,X})_p$ such that $w$ is not a zero-divisor of $\Oo _{X,p}$ and $(\Ii _{C,X})_p \cong w\Oo _{X,p}$ as a module over the local ring $\Oo _{X,p}$. We get that in a neighborhood of $p$ the divisor $C$ is a Cartier divisor of $X$. Let $I\subset \Oo _{X,p}$ be the ideal of $Y$ and $J\subset \Oo _{X,p}$ the ideal of $C$. We have $I \subset J$. First assume that $X$ is not reduced at a general point of $X$. Since the support of the nilradical $\eta \subset \Oo _X$  of the structural sheaf $\Oo_Y$ is a closed subset of $X_{\mathrm{red}}$, $X$ is not reduced at any point of $Y$ and in particular it is not reduced at $p$. Thus there is a nonzero $h\in I$ such that $h^m=0$ for some $m>0$. Since $I\subset J$, we have $h\in J$ and so $h$ is divided by $w$. Thus we get $w^m=0$ and so $w$ is a zero-divisor, a contradiction.

Now assume that $X$ is reduced at a general point of $Y$. Since $X$ is not integral and it has pure depth $n$, $X_{\mathrm{red}}$ has at least one another irreducible component. Since $h^0(\Oo _X)=1$, $X$ is connected and so $F\ne \emptyset$. Fix any $x\in F$. Since $\Oo _{X,x}$ has depth $n\ge 2$, it is connected in dimension $\le n-1$, i.e. for any open neighborhood $W$ of $x$ in $X$ and any closed subscheme $V$ of $W$, there is a neighborhhod $U$ of $x$ in $W$ such that $U\setminus (U\cap V)$ is connected. Thus $F$ has pure dimension $n-1$. Since $C\in| \Oo _Y(e)|$, $C$ is a Cartier divisor of $Y$. Thus $C$ is a Cartier divisor of $X$ at all points of $C\setminus (C\cap F)$. Since $e>0$, $C$ is an ample divisor of $Y$. In particular, we get $F\cap C \ne \emptyset$. Fix $p\in F\cap C$. Any local equation $w$ of $C$ at $p$ vanishes on each irreducible component of $X_{\mathrm{red}}$ containing $p$, because $w$ is assumed to be a non-zero divisor of $\Oo _{X,p}$. There is at least one another irreducible component of $X_{\mathrm{red}}$ containing $p$, because $p\in F$.

Part (iii) is obvious.
\end{proof}

As a corollary of Proposition \ref{zz2} we get the following result, which shows that $X$ is of wild representation type in a very strong form.

\begin{proposition}\label{zz3}
Take $X$ as above. For a fixed integer $w>0$, there is an integral quasi-projective variety $\Delta$ and a flat family $\{\Ff _a\}_{a\in \Delta}$ of aCM sheaf on $X$ with each $\Ff _a$ locally free outside a one-codimensional subscheme $C_a$ and for each $a\in \Delta$ the set of all $b\in \Delta$ such that
$\Ff _b \cong \Ff _a$ is contained in an algebraic subscheme $\Delta _a\subset \Delta$ with $\dim \Delta -\dim \Delta _a\ge w$.
\end{proposition}

\begin{proof}
First assume that $X$ has at least one irreducible component $Y$ with multiplicity at least $2$. Fix a positive integer $e$ such that $\dim |\Oo _Y(e)| \ge w$ and take as $\Delta$ the family of all integral $C\in |\Oo _Y(e)|$. Then we may apply (i) of Proposition \ref{zz2}. In this case we may find $\Delta$ with the additional condition that for all $a,b\in \Delta$ we have $\Ff _a\cong \Ff _b$ if and only if $a=b$.

Now assume that each irreducible component of $X$ has multiplicity $1$ and fix one of them, say $Y$. Write $F\subset Y$ as in (ii) of Proposition \ref{zz2}. Fix an integer $e >0$ such that $h^0(\Oo _X(e)) -h^0(\Oo _X(e)(-F) )> w$ and let $\Delta$ be the set of all integral divisors $C\in |\Oo _X(e)|$ not contained in $F$
and such that the scheme $F\cap C$ is reduced. Since $F$ has pure dimension $n-1$
and $C$ is an ample divisor,
the set $(F\cap C)_{\mathrm{red}}$ has pure dimension $2$. Note that if $C, D\in \Delta$ and $(C\cap F)_{\mathrm{red}} = (D\cap F)_{\mathrm{red}} $, then any equation of $C$ in $H^0(\Oo _X(e))$ differs from an equation of $D$ by an element of $H^0(\Oo _X(e)(-F))$. Then we may apply (ii) of Proposition \ref{zz2}.
\end{proof}


\bibliographystyle{amsplain}
\providecommand{\bysame}{\leavevmode\hbox to3em{\hrulefill}\thinspace}
\providecommand{\MR}{\relax\ifhmode\unskip\space\fi MR }
\providecommand{\MRhref}[2]{%
  \href{http://www.ams.org/mathscinet-getitem?mr=#1}{#2}
}
\providecommand{\href}[2]{#2}

\end{document}